\documentclass[10pt]{amsart}
\usepackage{amsfonts,amssymb,amsmath, forest}
\usepackage{amssymb, amsmath}
\usepackage{float,epsfig}
\usepackage{xcolor}
\usepackage{ mathrsfs }
\usepackage{graphicx}
\usepackage{hyperref}
\usepackage{mathabx}
\usepackage[left=2.5cm,top=1 cm,right=2.5cm,nofoot]{geometry}

\usepackage{cleveref}
\usepackage[mathlines]{lineno}

\newtheorem{thm}{Theorem}[section]

\newtheorem{deff}[thm]{Definition}

\newtheorem{prop}[thm]{Proposition}

\newtheorem{theorem}{Theorem}[section]

\newtheorem{example}[theorem]{Example}

\newenvironment{defn}{\begin{deff}
		\rm }{\end{deff}}
\newcommand{\mc}{\mathcal}

\renewcommand{\ss}{\subseteq}

\newcommand{\ra}{\rightarrow}
\newcommand{\msc}{\mathscr}
\newcommand{\ol}{\overline}
\newcommand{\G}{\mathscr G}

\def \H{\mathcal H}
\def \l2x{L^2(X;\mc H)}

\def \dx{{d{\mu_{X}}(x)}}

\DeclareMathOperator*{\range}{range}

\DeclareMathOperator*{\ess-sup}{ess-sup}
\DeclareMathOperator*{\Span}{span}

\theoremstyle{definition}

\theoremstyle{remark}

\numberwithin{equation}{section}

\usepackage{scalerel,stackengine}
\stackMath
\newcommand\reallywidehat[1]{%
	\savestack{\tmpbox}{\stretchto{%
			\scaleto{%
				\scalerel*[\widthof{\ensuremath{#1}}]{\kern.1pt\mathchar"0362\kern.1pt}%
				{\rule{0ex}{\textheight}}%WIDTH-LIMITED CIRCUMFLEX
			}{\textheight}% 
		}{2.4ex}}%
	\stackon[-6.9pt]{#1}{\tmpbox}%
}
\begin{document}
\sloppy	
\title{ An Application of the Supremum cosine angle between\\ multiplication invariant  spaces in $L^2(X; \mc H)$}
%\title{Injectivity of Sampling operator using supremum cosine angle for MI spaces}
\author{Sudipta Sarkar}

\author{Sahil Kalra}
\address{Department of Mathematics,
	Indian Institute of Technology Indore,
	Simrol, Khandwa Road,
	Indore-453 552}
\email{phd1701141004@iiti.ac.in, phd2001141005@iiti.ac.in, nirajshukla@iiti.ac.in}

\thanks{Research of S. Sarkar and Sahil Kalra was supported by research grant from CSIR, New Delhi [09/1022(0037)/2017-EMR-I] and University Grant Commision- Ref. No.: 191620003953, respectively.}

\author{Niraj K. Shukla}


%%%%%%%%%%%%%%%%%%%%%%%%%%%%%%%%%%%%%%%%%%%%%%%%%%%%%%%
%    General info

%%%%%%%%%%%%%%%%%%%%%%%%%%%%%%%%%%%%%%%%%%%%%%%%%%%%
\begin{abstract}
	
	In this article, we describe  the supremum cosine angle between two multiplication invariant (MI) spaces and its connection with the closedness of the sum of those spaces. The results obtained for MI spaces are preserved by the corresponding fiber spaces almost everywhere. Employing the Zak transform, we obtain the results for translation invariant spaces on locally compact groups by action of its closed abelian subgroup. Additionally, we provide the application of our results to sampling theory.
\end{abstract}

%%%%%%%%%%%%%%%%%%%%%%%%%%%%%%%%%%%%%%%%%%%%%%%%%%

%%%%%%%%%%%%%%%%%%%%%%%%%%%%%%%%%%%%%%%%%%%%%%%%%%
% Contents and Chapter Making
\maketitle
 
%
%\tableofcontents
%
%\clearpage

%%%%%%%%%%%%%%%%%%%%%%%%%%%%%%%%%%%%%%%%%%%%%%%%%

\section{{\bf Introduction }}

Let $E$ and $F$ be two closed subspaces of a separable Hilbert space $\mc H$, then the  \textit{supremum cosine angle} between $E$ and $F$ is defined by
\begin{equation}\label{eq:sup}
	\mathfrak{S}(E,F):=\mbox{sup}\left\{\frac{\|P_{F}u\|}{\|u\|}:{u\in E \backslash \{0\}}  \right\}=\|P_{F}|_{E}\|,
\end{equation}
where $P_{F}$ is  the orthogonal projection of $\mc H$ onto $F$ and $P_{F}|_{E}$ is its restriction on $E$. It is clear from the definition that
$\mathfrak S(E,F)=0$ if either $E = \{0\}$ or $F = \{0\}.$ In general, $\mathfrak S(E,F)=\mathfrak S(F,E)$.  The supremum cosine angle has applications in sampling theory and can be derived from the spectral coherence function used for sampling of non-ideal acquisition devices \cite{unser1994general}.  The
supremum cosine angle is connected to the closedness of the sum of two closed subspaces of a Hilbert space \cite{aldroubi1998construction,bownik2004biorthogonal,kim2003quasi,kim2005infimum,kim2008internal,Tang2000obliqueprojection}.
In \cite{Kim2002onrieszwavelet},  Kim et al. find out the conditions under which the sum of two singly generated shift-invariant subspaces of $L^2(\mathbb R^d)$ is closed. Later Kim et al. \cite{kim2006supremum}    generalized the result for shift invariant subspaces of $L^2(\mathbb R^d)$ with multi generators.
The shift invariant space has been widely used in sampling, wavelets, harmonic analysis, approximation theory,  and signal processing. 

 Our goal is to find the supremum cosine angle between multiplication invariant (MI) spaces using range functions and study the conditions under which the two MI spaces will be closed. The MI spaces were first introduced by M. Bownik and K. Ross \cite{bownik2015structure}. The study of such spaces is associated with the study of the Bessel system generated by multiplications in $\mc H$-valued Bochner space $L^2(X;\mathcal{H}).$ The idea is to generalize the modulation invariant subspaces of LCA groups and the same is done by introducing the concept of Parseval determining set in $L^1(X),$ which is the generalization of characters on an LCA group. The main advantage is that the results developed for the supremum cosine angle between  MI spaces will be utilized to find the supremum cosine angle between translation invariant subspaces of locally compact groups using Zak transform which converts the translation into multiplications \cite{iverson2015subspaces}. This setup covers all the classical ones like shift-invariant systems of $L^2(\mathbb R^d)$. The main reason to choose the MI spaces is that not only we can develop similar results for shift invariant spaces but also for translation invariant spaces where the translations are taken over a  non-discrete subgroup of LCA group using the Zak transformation. In this regard, it generalizes the classical results related to the closedness of two $\Gamma$-translation invariant subspaces in the context of a locally compact group $\mathscr G$ (need not be abelian), having $\Gamma$ as a closed abelian subgroup. One of the benefits of Zak transform for the pair $(\G, \Gamma)$ is that the various inaccessible examples pairs like        $(\mathbb R^n,  \mathbb R^m)$, $(\mathbb R^n, \mathbb Z^m)$,  $(\mathbb Q_p, \mathbb Z_p)$, $(\mathcal G, \Gamma)$, etc., can be accessed through it where  $n\geq m$,  $\Gamma$ (not necessarily co-compact, i.e., $\mathcal G/\Gamma$-compact, or uniform lattice) is a  closed subgroup of the second countable locally compact abelian (LCA) group $\mathcal G$, and   $\mathbb Z_p$ is the  $p$-adic integer in the $p$-adic number $\mathbb Q_p$. In case of $\Gamma$  co-compact, we have a similar characterization using fiberization. First, we characterize the supremum cosine angle using Gramian and dual Gramian operator (Theorem \ref{T:sup-pointwise}).  The Gramian and the dual Gramian analysis is more appropriate for the analysis of a Riesz basis and frame respectively. The Gramian analysis is more helpful for the case of finitely many generators as then the MI spaces can be characterized nicely using a finite matrix instead of an infinite matrix for the case of infinitely many generators. 
 
Next, we connect our results pointwise using range functions (Theorem \ref{T:sup-pointwise} and \ref{T:closedness}). The range functions are generally employed for the analysis of the local behavior of the MI spaces started by Helson \cite{helsonlectures}. For the setting of MI spaces, the global and local behavior of a frame agrees well with each other. Bownik et.al. in paper \cite{bownik2015structure} proved 
that the system generated by multiplications is a frame (Bessel) if and only if the corresponding local system is a frame (Bessel) almost everywhere. The MI operators has a one-to-one correspondence with range operators acting locally over fibers indexed by $x \in X$. Furthermore, many properties such as unitary equivalence of the MI systems, dual frame, and orthogonality for a pair of frames are also preserved. The reader can refer to \cite{bownik2019multiplication,iverson2015subspaces} for more details on range functions.

The organization of the paper is as follows: In Section \ref{S:Preliminaries}, we briefly discuss the preliminaries used in this paper. In Section \ref{S:Results}, we measure the supremum cosine angle between two MI spaces. Further, we obtain a   necessary and sufficient condition for the sum of two MI spaces to become closed. In Section \ref{S:Sampling} we establish a  condition so that the sampling operator becomes one-to-one. The paper ends with an application to the translation invariant spaces on locally compact groups by the action of its closed abelian subgroup in Section \ref{LC Group}. 

\section{Multiplication invariant spaces and range functions}\label{S:Preliminaries}
Throughout the paper, we fix some notations that are used throughout this paper. We denote $(X, \mu_X)$ to be a  $\sigma$-finite and complete measure space such that $L^2 (X)$ is separable and $\mc H$ is a separable Hilbert space. We briefly discuss the  supremum cosine angle for multiplication invariant spaces in $L^2(X; \mc H)$, where 	$$
L^2(X;\mathcal H)=\left\{\varphi \, |\  \varphi : X \rightarrow \mathcal H ~ \text{is measurable such that}\   \int_{X}\|\varphi(x)\|^2\  \dx<\infty \right\}.
$$ Now we define the multiplication operator  on $L^2(X;\mathcal H)$.
\begin{defn}
	For $\phi\in L^\infty(X)$, the operator $M_\phi$ on  $L^2(X;\mathcal H)$ is defined  by
	$$
	(M_\phi f)(x)=\phi(x) f(x), \ a.e. \ x \in X, \ f \in  L^2(X;\mathcal H),$$ is known as  \textit{multiplication operator}.
\end{defn}
The operator  $M_\phi$ is a bounded linear operator on $L^2 (X; \mc H)$ satisfying $\|M_\phi\|=\|\phi\|_{L^\infty}$ provided $X$  is a $\sigma$-finite measure space.  If $X$  is not a $\sigma$-finite measure space, then $\|M_\phi\|$ need not be same as $\|\phi\|_{L^\infty}$ \cite[Theorem 1.5]{conway1997course}. We define   Parseval determining set for $L^1 (X)$, introduced by Iversion \cite{iverson2015subspaces} which is a  measure-theoretic abstraction of characters for a locally compact abelian group.
\begin{defn}\label{PDS}
	Let  $(\mc M, \mu_{\mc M})$ be a $\sigma$-finite measure space. A set $\mc D=\{g_s\in L^{\infty}(X): s \in \mc M\}$ is said to be  \textit{Parseval determining set} for $L^1(X)$ if   for each $f \in L^1(X)$,   $s\mapsto \int_{X}f(x)\overline{g_s (x)}\ d{\mu_{X}}(x)$ is measurable on $\mc M$ and  
	$$\int_{\mc M}\bigg|\int_{X}f(x)\overline{g_s(x)} \ d{\mu_{X}}(x)\bigg| ^2d{\mu_{\mc M}}(s)=\int_{X} | f(x) |^2 \  d{\mu_{X}}(x).$$  
	
\end{defn}

Further, we define multiplication invariant space on $L^2(X;\mc H)$ associated with a Parseval determining set.
\begin{defn} \label{MIOpDefn}
	Let $V$ be a closed subspace of $L^2(X;\mathcal H)$ and $\mc D$ be a Parseval determining set for $L^1(X)$.  We say $V$  is a  \textit{multiplication invariant (MI) space} corresponding to $\mc D$ if  $$M_\phi f \in V, \  \mbox{for all} \ \phi \in \mc D \ \mbox{and} \ f \in V.$$
\end{defn}
\begin{defn}
	Given a family $\mc A \subset  \l2x$ and a Parseval determining set  $\mc D \subset L^\infty (X)$ for $L^1(X)$,  we define  \textit{multiplication generated (MG)}  system  
	$\mc E_{\mc D}(\mc A)$ and its associated MI  space
	$\mc S_{\mc D}(\mc A)$ as follows:
	\begin{equation}\label{eq:ED}
		\mc E_{\mc D}(\mc A) :=\{M_{\phi}\varphi: \phi\in \mc D,\varphi \in \mc A\} \ \mbox{and } \mc S_{\mc D}(\mc A) := \ol{\Span} \  \mc E_{\mc D}(\mc A),
	\end{equation}
	respectively. 
	The system $\mc E_\mc D(\mc A)$ is
	\begin{itemize}
		\item[(i)] \textit{complete}, if $\ol{\Span} \,\mc E_\mc D(\mc A)= L^2(X;\mc H);$
		\item [(ii)] a \textit{Bessel MG-system}, if there existes $M>0$ such that 
		$$\sum_{\phi \in \mc A}\int_X
		|<f(x), M_\phi \varphi(x)>|^2d\mu(x) \leq M\|f\|^2\,\,\,\,\,\,\,\text{for all $f\in \mc H;$}$$
		\item[(iii)] a \textit{frame} for $\mc S_\mc D(\mc A)$, if there exists constants $m, M>0$
		$$m\|f\|^2\leq\sum_{\phi \in \mc A}\int_X
		|<f(x), M_\phi \varphi(x)>|^2d\mu(x) \leq M\|f\|^2\,\,\,\,\,\,\,\text{for all $f\in \mc S_\mc D(\mc A)$}.$$
	\end{itemize}
\end{defn}
			For the characterization of MI   spaces, the range function plays a crucial role. The history of the range function started from the work of Helson \cite{helsonlectures} and then widely been used for the characterization of translation invariant spaces and multiplication invariant spaces   \cite{bownik2015structure,iverson2015subspaces,bownik2019multiplication}.
			A \textit{range function}  on $X$ is a mapping 
			$J:X\rightarrow \{\text{closed subspaces of} ~\mathcal H\}$.
			Further, we say $J$  is  \textit{measurable}  if   for any $u,v\in \mathcal H$ the mapping 
			$x\mapsto \langle P_J(x)u,v\rangle$ is measurable on $X$, where for $x\in X$, the  orthogonal projection $P_{J} (x): \mathcal H \rightarrow \mathcal H$  projects   onto $J(x)$.   
			We can associate a closed subspace $V_J$ with every  projection-valued  map $J$ as follows:
			\begin{align}\label{V_j}
				V_J:=\left\{\varphi \in L^2(X;\H):\varphi(x)\in J(x)~\mbox{for}~ a.e.~ x\in X\right\}.
			\end{align}
			It is clear that $V_J$ is an MI space. The following characterization of MI spaces in terms of range functions suggests that every MI space has the form given in (\ref{V_j}).
			\begin{thm}\cite{bownik2019multiplication}
				The following statements are equivalent for a closed subspace $V \subset L^2(X;\mc H)$ and Parseval determining set $\mc D:$
				\begin{itemize}
					\item [$(i)$]The space V is MI.
					\item [$(ii)$] There exists a measurable range function $J$
					such that $V_J=V.$
					\item [$(iii)$] $M_\phi V \subset V$ for every $\phi \in \mc D.$
				\end{itemize}
				Furthermore, if $V$ is an MI space generated by an countable collection $\mc A=\{\psi_i\}_{i\in \mathbb{N}},$ i.e., $V=\ol{\Span}\{M_\phi\psi_i : \phi \in \mc D, i \in \mathbb{N}\} ,$ then the corresponding range function $J:=J_\mc A$ is given by
				$$J_\mc A(x) = \ol{\Span}\{\psi_i(x) : i \in \mathbb{N}\}\,\,\,\,\text{ for a.e. $x\in X$}.$$
			\end{thm}
			For a countable family of functions $\mc A=\{\psi_i: i\in I\} \subset L^2(X;\mc H),$ let the multiplication generated system $ \mc E_{\mc D}(\mc A)$ be Bessel. Then for a.e. $x\in X$, the system     $\{\psi_i(x): i\in I\}$ is Bessel \cite[Theorem 5.5]{bownik2019multiplication}. The \textit{analysis operator} $T_{\mc A}(x):  J_{\mc A} (x)\rightarrow\ell^2(I)  $ of $\{\psi_i(x): i\in I\}$ is given by 
			$$T_{\mc A}(x)(h)= \{\langle h, \psi_i(x) 	\rangle\}_{i\in I}\,\,\,\, \text{for every $h \in J_{\mc A}(x)$}.$$
			The \textit{synthesis operator} of  $\{\psi_i(x): i\in I\}$ is the adjoint of analysis operator $T_{\mc A}(x)$ and is given by 
			$$T_{\mc A}^*(x) : \ell^2(I) \ra J_{\mc A} (x),\,\ T_{\mc A}^*(x) (c)=\sum_{i \in I} c_i \psi_i (x)$$
			for $c=\{c_i\}_{i \in I}$  having finitely many   non-zero terms.
			The operator $G_{\mc A}(x): =T_{\mc A}(x)T_{\mc A}^*(x)$  on $\ell^2(I)$ is the  \textit{Gramian}  of $\{\psi_i(x): i\in I\}.$ Let $\mc A'=\{\varphi_i\}_{i \in I} \subset L^2(X;\mc H)$  be another countable collection such that $ \mc E_{\mc D}(\mc A')$ is also Bessel, then the operator $G_{\mc A,\mc A'}(x): = T_{\mc A}(x)T_{\mc A'}^*(x)$  is the \textit{mixed Gramian} of $\{\psi_i(x)\}_{ i\in I}$ and $\mc \{\varphi_i(x)\}_{i \in I}$ \cite{christensen2016introduction}.  The matrix representations of Gramian and mixed Gramian operator corresponding to the collections $\mc A=\{\psi_i: i\in I\}$ and  $\mc A'=\{\varphi_i: i\in I\}$, is 
			\begin{align*}
				G_{\mc A}(x)=\left[\langle \psi_j(x),\psi_i(x)\rangle\right]_{i, j \in I} \  \mbox{and} \ G_{\mc A,\mc A'}(x)=\left[\langle \psi_j(x),\varphi_i(x)\rangle\right]_{i, j \in I}. 
			\end{align*}

\section{Supremum cosine angles for MI spaces}\label{S:Results}
In this section, we will measure the supremum cosine angle between two MI spaces in terms of their Gramian. Also, we find a necessary and sufficient condition for the closedness of the sum of two multiplication invariant spaces employing the supremum cosine angle. The characterization of the closedness of the sum of two subspaces of a separable Hilbert space in terms of the supremum cosine angle was first studied by Tang et. al. \cite{Tang2000obliqueprojection}. Due to the applications in wavelets and multi resolution analysis, the problem was shifted to singly generated shift invariant spaces by Kim at. al. \cite{kim2006supremum}. The similar characterization for MI spaces and the case of shift invariant spaces or more generally translation invariant spaces can easily be deduced out of it using Zak transform (\ref{zak trans}). 

The following result Theorem $(\ref{T:sup-pointwise}),$ which is our first main result, describes the supremum cosine angle between two multiplication invariant spaces in terms of Gramians of corresponding generating sets which an abstract version of  \cite[Lemma 3.6]{kim2006supremum}.
\begin{thm}\label{T:sup-pointwise} 
	Let  $\mc A=\{\varphi_i: i\in I\}$ and $\mc A'=\{\psi_i: i\in I\}$ be two countable collections of functions in $L^2(X;\mc H).$ Define $\sigma(S_\mc D(\mc A)):=\{x \in X : J_\mc A(x) \neq 0\}$ and $\Omega:= \sigma(\mc S_{\mc D}(\mc A)) \cap \sigma(\mc S_{\mc D}(\mc A')) $. Then the supremum cosine angle between $\mc  S_{\mc D}(\mc A)$ and $\mc  S_{\mc D}(\mc A')$ is given by
	$$\mathfrak S(\mc  S_{\mc D}(\mc A), \mc S_{\mc D}(\mc A'))=\ess-sup_{x\in \Omega}\big\{\mathfrak S\big(J_{\mc A}(x), J_{\mc A'}(x)\big)\big\}.$$
	In addition, if the multiplication generated systems $\mc E_{\mc D}(\mc A)$ and $\mc E_{\mc D}(\mc A')$ are  frames for $\mc S_{\mc D}(\mc A)$ and $\mc S_{\mc D}(\mc A')$, respectively, then		$$\mathfrak S(\mc  S_{\mc D}(\mc A), \mc S_{\mc D}(\mc A'))=\ess-sup_{x\in \Omega}\left\|(G_{\mc A'}(x)^\dagger )^\frac{1}{2} G_{\mc A, \mc A'}(x) (G_{\mc A}(x)^\dagger)^\frac{1}{2}\right\|,$$  where $\dagger$ denotes the pseudo inverse.
\end{thm}
Our second main result Theorem $(\ref{T:closedness})$ provides the condition under which the sum of two MI spaces will be closed in terms of supremum cosine angle. In addition to that, we have shown that the similar condition is being preserved by the corresponding fiber spaces almost everywhere. The result depicts that the sum of two MI spaces will be closed if and only if the supremum cosine angles between the corresponding fiber spaces less than 1, almost everywhere. This result is an abstract version of  \cite[Lemma 3.6]{kim2006supremum} for MI spaces. 
\begin{thm}\label{T:closedness}
	In addition to the standing assumptions as in Theorem $\ref{T:sup-pointwise}$, let
	$$\Omega^{'}:= \Omega \backslash\sigma \big(S_{\mc D}(\mc A) \cap S_{\mc D}(\mc A')\big).$$
	Then the following are equivalent:	
	\begin{enumerate}
		\item[(i)]  	
		$\mc S_{\mc D}(\mc A)|_{\Omega^{'}}+\mc S_{\mc D}(\mc A')|_{\Omega^{'}}$ is closed. 
		\item[(ii)]
		$\mathfrak S(\mc S_{\mc D}(\mc A)|_{\Omega^{'}},\mc S_{\mc D}(\mc A')|_{\Omega^{'}})<1$.
		\item[(iii)]
		$J_{\mc A}(x)+J_{\mc A'}(x)$ is closed for 
		a.e. $x\in \Omega'$.
		\item[(iv)]
		$\mathfrak S(J_{\mc A}(x), J_{\mc A'}(x))<1$ for 
		a.e. $x\in \Omega'$.
	\end{enumerate}
\end{thm}

%
%The following Lemma expresses that the mixed dual Gramian for the fibers can be expressed in terms of a projection map.
%\begin{lemma}\label{L:mixedGramian}Given two collections let $\mc A=\{\varphi_i:i\in I\}$ and $\mc A'=\{\psi_i:i\in I\}$ in $\mc H$. let the fibers
%	$\mc A(x)=\{F\varphi_i(x): i \in I\}, \mc A'(x)=\{\psi_i: i \in I\}$ be the Bessel systems. Then then the mixed dual Gramian for $\mc A(x)$ and $\mc A'(x)$ is $G_{\mc A'(x), \mc A(x)}=T_{\mc A'(x)}PT_{\mc A(x)}$, where $P=P_{\mc A'(x)|_{\mc A(x)}}.$
%\end{lemma}
%\begin{proof}
%	Calculating the following, for $\beta\in J_{ \mc A(x)}=\ol{\mbox{span}}\{F\varphi_i(x):i \in I\}$, we have, 
%	\begin{align*}
	%T_{\mc A'(x)}PT_{\mc A(x)}\beta
	%&=T_{\mc A'(x)}^*P(\sum_{j\in I }c_j\varphi_j(x)
	%=T_{\mc A'(x)}^*(\sum_{j\in I }c_jP\varphi_j(x))\\
	%	&=\left\langle\sum_{j\in I}c_jP\varphi_j(x), F\psi_i(x)\right\rangle_{i\in I}
	%	=\left(\sum_{j=1 }^{\infty}c_j\left\langle P\varphi_j(x), F\psi_i(x)\right\rangle\right)_{i\in I}\\
	%	&=\left(\sum_{j\in I }c_j\left\langle P_{\mc A'(x)} \varphi_j(x), F\psi_i(x)\right\rangle\right)_{i\in I}\\
	%	&=\left(\sum_{j\in I}c_j\left\langle \varphi_j(x),P_{\mc A'(x)}  F\psi_i(x)\right\rangle\right)_{i\in I}\\
	%	&=\left(\sum_{j\in I}c_j\left\langle \varphi_j(x),  F\psi_i(x)\right\rangle\right)_{i\in I}\\
	%	&=G_{\mc A'(x), \mc A(x)}\beta
	%	\end{align*}
%	
%\end{proof}

Before proving the above results, we first deduce a couple of propositions. The following Proposition  $(\ref{P:Spectrum-Property})$ is an abstract version of \cite[Lemmas 3.3 and 3.4]{cabrelli2009samplingoperator}.
\begin{prop}\label{P:Spectrum-Property}
	Let  $\mc A=\{\varphi_i: i\in I\}$ and $\mc A'=\{\psi_i: i\in I\}$ be two countable collections of functions in $L^2(X;\mc H).$ For the MI spaces  $S_{\mc D}(\mc A)$ and $S_{\mc D}(\mc A')$  in $L^2(X; \mc H)$,  the following hold true: 
	\begin{enumerate}
		\item[(i)]	If  $S_{\mc D}(\mc A)\cap S_{\mc D}(\mc A')=\{0\}$, then $J_{\mc A}(x)\cap J_{\mc A'}(x)=\{0\}$ for a.e. $x \in X$.
		\item[(ii)]  $\sigma\big(S_{\mc D}(\mc A)\cap S_{\mc D}(\mc A')\big)=\{x \in X: J_{\mc A}(x)\cap J_{\mc A'}(x)\neq \{0\}\}\ss \sigma(S_{\mc D}(\mc A))\cap \sigma(S_{\mc D}(\mc A'))$.
	\end{enumerate}
\end{prop}
\begin{proof}
	(i)	On the contrary, let us assume $J_{\mc A}(x)\cap J_{\mc A'}(x)\neq \{0\}$ for a.e. 
	$x \in X$, then there exists a measurable set $\Delta \subset X$ such that $\mu_{X}({\Delta})>0$ and for all $x\in \Delta$, $J_{\mc A}(x)\cap J_{\mc A'}(x)\neq \{0\}$. Let $\{d_1, d_2,\dots, d_n,\dots\}$ be a countable dense subset of $\mc H$ and $U(x)=J_{\mc A}(x)\cap J_{\mc A'}(x)$. Since $U(x)\neq \{0\}$ on $\Delta,$ then for each $x\in \Delta$, there exists $d_n$ such that $P_{U(x)}d_n\neq 0$. Let $i(x)$ be the least positive integer for which $P_{U(x)d_{i(x)}}\neq 0$. Consider the set $\Delta_n=\{x\in \Delta: i(x)=n\}$, then it is clear that $\Delta=\oplus_{n \in \mathbb N}\Delta_n$. Since $\mu_{X}(\Delta)>0,$ there exists $N\in \mathbb N$ such that $\Delta_N$ has positive measure. Define a function $f\in L^2(X; \mc H)$ such that : $$f(x)=\begin{cases}
		P_{U(x)}d_N, \ \mbox{when}  \ x\in \Delta_N,\\
		0, \  \mbox{otherwise}.
	\end{cases}$$
	Then using  \cite[Theorem 3.8]{bownik2015structure}
	, $0\neq f\in S_{\mc D}(\mc A)\cap 
	S_{\mc D}(\mc A')$,  a contradiction.
	
	(ii) We prove that $\sigma\big(S_{\mc D}(\mc A)\cap S_{\mc D}(\mc A')\big)=\{x \in X: J_{\mc A}(x)\cap J_{\mc A'}(x)\neq \{0\}\}$ as the later containment of sets is quite obvious. Let $\msc W:=\mc S_{\mc D}(\mc A)\cap \mc S_{\mc D}(\mc A')$, $E':=\mc S_{\mc D}(\mc A)\ominus  \msc W$ and $F'=\mc S_{\mc D}(\mc A') \ominus \msc W$, then $\msc W,E', F'$ are multiplication invariant  spaces, and $\mc S_{\mc D}(\mc A)=E'\oplus \msc W$ and $\mc S_{\mc D}(\mc A')=F'\oplus \msc W$. Moreover, $ J_{\mc A}(x)=E'(x)\oplus\msc W(x)$ and $J_{\mc A'}(x)=F'(x)\oplus\msc W(x)$, by the pointwise projection property of $\mc D$-multiplication spaces. Now if $x\in \sigma(\mc S_{\mc D}(\mc A)\cap \mc S_{\mc D}(\mc A'))=\sigma(\msc W)$, then $\{0\}\neq \msc W(x)\ss J_{\mc A}(x)\cap J_{\mc A'}(x)$.
	
	Conversly, let $x \in X$ is such that $J_{\mc A}(x)\cap J_{\mc A'}(x)\neq \{0\}.$ Then, there exists 
	$0 \neq h \in J_{\mc A}(x)\cap J_{\mc A'}(x).$ Also there exists $h_1 \in E'(x), h_2 \in F'(x)$ and $w_1, w_2 \in \msc W(x)$ such that $h = h_1 + w_1 = h_2 + w_2.$ Since $h_1-h_2 \in \msc W(x)^\perp$ 
	and $w_1-w_2 \in  \msc W(x)$ implies $h_1 =h_2$ and $w_1 =w_2.$ Therefore, $h = h_1 + w_1$ with $h_1 \in E'(x) \cap F'(x).$ Also it is clear that, $E' \cap F'= \{0\}$ which further by Proposition \ref{P:Spectrum-Property} implies $E'(x) \cap F'(x) =\{0\}.$ Hence, $h_1=0$ and $0 \neq h \in \msc W(x) = \big(\mc S_{\mc D}(\mc A)\cap \mc S_{\mc D}(\mc A')\big)(x).$
%It is easy to see that $\{x\in X: J_{\mc A}(x)\cap J_{\mc A'}(x)\neq \{0\}\}\ss \sigma(\mc S_{\mc D}(\mc A))\cap \sigma(\mc S_{\mc D}(\mc A'))$. 
%	 Suppose on the other hand there exist $a\in \mc H\backslash\{0\}$ such that $a\in E(x)\cap F(x)$. Then there exist $b\in E'(x), b'\in F'(x)$ and $c, c'\in \msc W(x)$ such that $0\neq a=b+c=b'+c'$. Therefore $b-b'=c-c'$. $b-b'\in \msc W(x)^\perp$ and $c-c'\in \msc W(x)$ we have $0=b-b'=c'-c$, which implies $0\neq a=b\oplus c$ and $b\in E(x)\cap F(x)$ and $c\in \msc W(x)$. Also note that 
%	 $$E'\cap F'=[E-(E\cap F)]\cap[F-(E\cap F)]=(E\cap F)-(E\cap F)=\{0\}$$
%	 Lemma \ref{L:Zerospect} implies that $E'(x)\cap F'(x)=\{0\}$ a.e. Therefore $b=0$. Hence there exist a non-zero vector $a=c\in \msc W(x)=(E\cap F)(x)$. This implies for a.e. $x\in X$, $E(x)\cap F(x)\neq \{0\}$, $x\in\sigma(E\cap F).$
\end{proof}
The next proposition proves that the restriction of a  multiplication invariant space to a measurable subset is again multiplication invariant.
The following result is an abstract version of  \cite[Lemma 2.5]{kim2006supremum}.
\begin{prop}\label{Restriction-space} For a countable 
	collection of functions $\mc A$ in $L^2(X;\mc H),$ consider the MI space $S_{\mc D}(\mc A)$  and its restriction to a measurable subset  $\Lambda$ of 
	$X$ defined by, $$S_{\mc D}(\mc A)|_{\Lambda}:=\{f\in S_{\mc D}(\mc A): f(x)=0\  \mbox{a.e. on} \  X\backslash \Lambda\}.$$ Then $S_{\mc D}(\mc A)|_{\Lambda}$ is also an MI space.
\end{prop}
\begin{proof}
	Let $\{u_n\}_{n\in \mathbb N}$ be a  sequence in $S_{\mc D}(\mc A)|_{\Lambda}$ such that $u_n\longrightarrow u$ in $L^2(X;\mc H)$ as $n \longrightarrow \infty$ for some $u\in L^2(X;\mc H)$. Since $S_{\mc D}(\mc A)$ is closed and $S_{\mc D}(\mc A)|_{\Lambda}\ss S_{\mc D}(\mc A)$, we have $u\in S_{\mc D}(\mc A)$. Also note that,  supp $(u_n)=\Lambda$ implies supp $(u)=\Lambda$. Hence, $u\in S_{\mc D}(\mc A)|_{\Lambda}.$ Therefore $S_{\mc D}(\mc A)|_{\Lambda}$ is closed and an MI space.
\end{proof}

	We are now ready to prove our main results Theorems (\ref{T:sup-pointwise}) and  (\ref{T:closedness}).
	
	\begin{proof}[Proof of Theorem \ref{T:sup-pointwise} ]
		
		Let $\{d_n:n\in \mathbb N\}$ be an orthonormal basis of $L^2(X;\mc H)$. For each $n \in \mathbb N,$ define a function $s_n:X\rightarrow \mathbb R, x\mapsto s_n(x):=\langle P_{J_{\mc A}(x)}|_{J_{\mc A'}(x)} d_n, d_n\rangle.$ By definition of range function, the sequence $s_n$ is measurable. Denote $s: = \sup_{n \in \mathbb{N}} s_n$ defined by, $x\mapsto \mathfrak S( J_{\mc A}(x),  J_{\mc A'}(x))=\sup_{n \in \mathbb{N}}\langle P_{J_{\mc A}(x)}|_{J_{\mc A'}(x)} d_n, d_n\rangle=\left\|P_{ J_{\mc A'}(x)}|_{ J_{\mc A}(x)}\right\|.$ Then the map $s$ is weakly measurable, where the last equality holds because projection maps are self adjoint. Let $c=\mathfrak S(S_{\mc D}(\mc A),S_{\mc D}(\mc A'))=\|P_{S_{\mc D}(\mc A')}|_{S_{\mc D}(\mc A)}\|$ and $\tilde c=\ess-sup_{x\in \Omega}\big\{\mathfrak S\big(J_{\mc A}(x), J_{\mc A'}(x)\big)\big\}$. For $u\in S_{\mc D}(\mc A)$, we have $\|P_{J_{\mc A'}(x)}u(x)\|\leq \mathfrak S(J_{\mc A}(x), J_{\mc A'}(x))\|u(x)\|,$   for a.e. $x \in X$ and hence 
		\begin{align*}
			\|P_{S_{\mc D}(\mc A')}(u)\|^2=\int_{X}\| P_{S_{\mc D}(\mc A')}(u)(x)\|^2\ \dx&= \int_{X} \|P_{J_{\mc A'}(x)}u(x)\|^2\ \dx=\int_{\Omega}\|P_{J_{\mc A'}(x)}u(x)\|^2\ 
			\dx\\
			&\leq \tilde{c}^2\int_{X}\|u(x)\|^2\ \dx \leq \tilde{c}^2 \|u\|^2,
		\end{align*}
		by observing $P_{S_{\mc D}(\mc A')}(u)(x) = P_{J_{\mc A'}(x)}u(x)$ given in   \cite[Proposition 2.2]{bownik2015structure}. Taking the supremum over all $u\in S_{\mc D}(\mc A)$ we have $c\leq \tilde c$.
		
		By choosing $c<\tilde{c},$ there exist  $\epsilon>0$ and a  measurable subset $\Delta\ss \Omega$, with $\mu_{X}{(\Delta)}>0$ such that $u+\epsilon<\mathfrak S(J_{\mc A}(x),J_{\mc A'}(x))$ for a.e. $x\in X$.  Construct $\Delta_i=\{x\in S_{\mc D}(\mc A): 0<(c+\epsilon)\|P_{J_{\mc A}(x)}d_i\|\leq \| P_{J_{\mc A'}(x)}(P_{J_{\mc A}(x)}d_i)\|\}.$
		Then $X=\bigcup_{i=1}^{\infty}\Delta_i$. Hence there exists $i_0$ such that $\mu_{X}(\Delta_{i_0})>0.$ Define a function $h$ by$$h(x)= \begin{cases}
			P_{J_{\mc A}(x)}d_i,\  \mbox{if  } x \in \Delta_{i_0},\\
			0, \  \mbox{otherwise}.
		\end{cases}$$
		Then, $h\in S_{\mc D}(\mc A)$ and hence the following calculation
		\begin{align*}
			\|P_{S_{\mc D}(\mc A')}h\|^2&=\int_{\Omega}\|P_{J_{\mc A'}(x)}h(x)\|^2\ \dx=\int_{\Delta_{i_0}}\|P_{J_{\mc A'}(x)}P_{J_{\mc A}(x)}d_{i_0}\|^2\ \dx\\
			&\geq (c+\epsilon)^2 \int_{\Delta_{i_0}}\|P_{J_{\mc A}(x)}d_{i_0}\|^2\ \dx=(c+\epsilon)^2\int_{ \Delta_{i_0}}\|h(x)\|^2\ \dx=(c+\epsilon)^2\|h\|^2>0,
		\end{align*}
		implies that $c=\mathfrak S(S_{\mc D}(\mc A),S_{\mc D}(\mc A'))\geq \frac{\|P_{S_{\mc D}(\mc A')}h\|}{\|h\|}\geq c+\epsilon$, a contradiction. So $c\geq \tilde c$. Hence $c=\tilde c$.
		
		For the remaining part, first observe that the fibers 
		$\{\varphi_i(x) : i \in I\}$
		and $\{\psi_i(x) : i \in I\}$ are frames for $J_\mc A(x)$ and $J_\mc A'(x)$ for a.e. $x \in X,$ since	$\mc E_{\mc D}(\mc A)$ and $\mc E_{\mc D}(\mc A')$ are frames for $\mc S_\mc D(\mc A)$ and $\mc S_\mc D(\mc A')$ by  \cite[Theorem 2.10]{iverson2015subspaces}. The hypothesis imply that $T_{\mc A}(x), T_{\mc A'}(x), T_{\mc A}^*(x),T_{\mc A'}^*(x), G_{\mc A}(x), G_{\mc A'}(x)$ have closed range and their pseudo inverses exist. Now, the result follows using \cite[Theorem 2.1]{kim2006supremum}. 
	\end{proof}
	Next, we provide a proof of Theorem \ref{T:closedness} by observing that for two closed  subspaces $E$ and $F$ of  $\mc H$,  the sum $E+F$ is closed and $E\cap F=\{0\}$ if and only if $\mathfrak S(E,F)<1$ \cite[Theorem 2.9]{Tang2000obliqueprojection}.
	\begin{proof}[Proof of Theorem \ref{T:closedness} ]
		(i) $\Leftrightarrow$ (ii) By Proposition \ref{Restriction-space},	$\mc S_{\mc D}(\mc A)|_{\Omega'}$ and $\mc S_{\mc D}(\mc A')|_{\Omega'}$ are multiplication-invariant spaces. Now applying Proposition \ref{P:Spectrum-Property}, $\sigma(\mc S_{\mc D}(\mc A)|_{\Omega'}\cap \mc S_{\mc D}(\mc A')|_{\Omega'})=\phi$ therefore $\mc S_{\mc D}(\mc A)|_{\Omega'}\cap \mc S_{\mc D}(\mc A')|_{\Omega'}=\{0\}$. Applying
		\cite[Theorem 2.9]{Tang2000obliqueprojection},
		(i) and (ii)
		are equivalent. The statements (iii) and (iv) are also equivalent by following the same arguments since for a.e. $x \in  \Omega'$, $J_{\mc A}(x)\cap J_{\mc A'}(x)=\{0\}.$ 
		
		The statements (ii) and (iv) are equivalent by noting  $\mathfrak S(\mc S_{\mc D}(\mc A)|_{\Omega^{'}},\mc S_{\mc D}(\mc A')|_{\Omega^{'}})=\ess-sup_{x\in \Omega'}\big\{\mathfrak S\big(J_{\mc A}(x), J_{\mc A'}(x)\big)\big\}$  using Theorem \ref{T:sup-pointwise} and  $\mathfrak S(\mc S_{\mc D}(\mc A)|_{\Omega^{'}},\mc S_{\mc D}(\mc A')|_{\Omega^{'}})<1$ if and only if  $\mathfrak S(J_{\mc A}(x), J_{\mc A'}(x))<1$ for 
		a.e. $x\in \Omega'$. 
	\end{proof}
	In the next section, we will provide an application of the supremum cosine angle for the sampling problem.
\section{Application to the Sampling theory}\label{S:Sampling}

 In this section, we discuss the application part of the of supremum cosine angle. In sampling theory, it is desirable for the sampling operator to be one-to-one for the unique representation of a function.  We find out the conditions under which the sampling operator defined for a Bessel MG system is one-to-one.  We also connect our result with the injectivity of the pointwise sampling operator defined on each fiber space.  The results obtained in this section are an abstract version of results already obtained for closed subspaces of a separable Hilbert space or more generally shift invariant spaces \cite{cabrelli2009samplingoperator,ludo2008samplingoperator}.
We begin with some definitions.
\begin{defn} 
	Let $(\mathcal{M},\mu_\mathcal{M})$ be a $\sigma$-finite measure space and $\mathcal{A} =\{\psi_i\}_{i \in 
		I}\subset L^2(X;\mathcal{H})$ be a set of generators such that the MG system $\mc E_\mathcal{D}(\mathcal{A})$ is Bessel.	
	Then the \textit{sampling operator} associated with the system  $\mc E_\mathcal{D}(\mc A)$ is defined by
	\begin{equation}
		T: L^2(X;\mathcal{H})  \ra L^2( \mc M \times I),\ Tf:= \{\langle f,M_{g_{t}}\psi_i\rangle\}_{t\in \mc M, i \in I}.	
	\end{equation}	
\end{defn}
With every sampling operator $T$, we can associate pointwise sampling operator $\tilde{T}(x)$ of  Bessel family $\{\psi_i(x)\}_{i \in I}$ defined on the fibre space $J_\mc A(x)$ for a.e. $x \in X$ and is given by 
$$\tilde{T}(x): J_\mc A(x) \rightarrow \ell^2(I), \ \tilde{T}(x)(\gamma) = \{<\gamma,\psi_i(x)>\}_{i \in I}. $$
The sampling operator $T$ associated with the Bessel system $\mc E_\mc D(\mc A)$ is the analysis operator $T_\mc A$ of the Bessel system $\mc E_\mc D(\mc A)$. The requirement is to obtain conditions for the sampling operator to be one-to-one to uniquely obtain a signal from its samples. The following result is an abstract version of \cite[Proposition 5.11]{cabrelli2009samplingoperator} for the setting of MI spaces. Additionally, it provides the equivalent conditions for the sampling operator $T$ to be one-to-one in terms of injectivity of the  pointwise sampling operators $\tilde{T}(x)$ for a.e. $x \in X.$ 
%%%%%%%%%%%%%%%%%%%%%%%%%%
% The reconstruction of a signal $f\in \mc H$, from a set of data $\mc A=\{\psi_i\}_{i\in I}$, is one of the main concerns in sampling. The set $\mathcal A$ is called the sampling set. For the unique representation we need the sampling operator $T$ to be one-to-one. Throughout the section, we will discuss the one-to-one property of the sampling operator $T$ using the supremum cosine angle. The study was started in \cite{ludo2008samplingoperator} and later on, \cite{cabrelli2009samplingoperator} for the shift-invariant spaces. Now we will state a measure-theoretic abstraction of the Theorem stated in  \cite[Theorem .......]{ludo2008samplingoperator}, which describes the one-to-one property of the sampling operator on MI spaces.
%\begin{defn}
%A sampling operator $A$ is called stable on $S$ if there exists $0<x\leq \beta <\infty$ such that 
%$$x \|x_1-x_2\|^2 \leq \|Ax_1-Ax_2\|^2\leq \beta \|x_1-x_2\|^2 \forall x_1, x_2\in \mc S.$$
%\end{defn}

\begin{prop}\label{SamplingOp1-1}
	
	Let  $A=\{\psi_i\}_{i\in I}\subset L^2(X; \mc H)$ be a countable family of functions and $\mc D$ be a Parseval determining set such that $\mc E_{\mc D}(\mc A)$ is  Bessel in $L^2(X;\mc H)$ and $T$ be the associated sampling operator. Then for a finite set  $\mc A'=\{\phi_1,\phi_2,...,\phi_r\} \subset L^2(X; \mc H), $ the following are equivalent:
	\begin{enumerate}
		\item[(i)]   The sampling operator $T$ is one-to-one on $S_\mathcal{D}(\mathcal{A'})$.
		\item[(ii)] $\mc E_{\mc D}(\mc A)$ is complete in $S_\mc D(\mc A') $ i.e. $S_\mc D(\mc A')=\ol{\Span }   \  \mc E_{\mc D}(\mc A)$.
		\item[(iii)]  $\tilde{T}(x)$ is one-to-one   on $J_{\mathcal{A'}}(x)$ for a.e. $x\in X$.
		\item[(iv)] For a.e. $x \in X,$ the system $\{\psi_i(x): i \in I\}$ is complete in $J_\mc  A(x).$
		\item[(v)] $\ker (G_{\mc A', \mc A})(x)=\ker( T_{\mc A'}^*(x))$ for a.e. $x \in X$. 
		\item[(vi)] $\dim (\range(G_{\mc A', \mc A})(x))=\dim J_{\mc A'}(x)$ for a.e. $x \in X$.	\end{enumerate}
\end{prop}
\begin{proof}
	The sampling operator $T$ is one-to-one on $S_\mc D(\mc A')$ if and only if for all $t \in \mc M \text{ and } i \in I, <f, M_{g_t}\psi_i> = 0$  will imply $f=0$ for any $f \in S_\mc D(\mc A')$ which is equivalent to $(ii).$ 
	Therefore statements (i) and (ii) are equivalent. Similarly, statements (iii) and (iv) are equivalent. Also, if (ii) is true,  then by \cite[Theorem 3.5]{bownik2019multiplication}, the corresponding range function $J_{\mc A'}$ satisfies $$J_{\mc A'} (x)=\overline{\Span} \{\psi_i(x) : i \in I\}.$$
	Since $J_{\mc A'}(x)$ is closed, then by \cite[Proposition 3.2]{cabrelli2009samplingoperator}, $\tilde{T}(x)$ is one-to-one on $J_{\mc A'}(x)$ for a.e. $x \in X$ and hence the  statements (ii) and (iii)
	are equivalent. The equivalence of the statements (i), (v), and (vi) follows from the fact that the fiber space $J_{\mc A'}(x)$ is finite dimensional and the range of the synthesis operator $T_{\mc A'}^*(x)$  is $J_{\mc A'}(x)$.  %%%%%%%%%%%%%%%%%%%%%%%%%
%(i) $\Longleftrightarrow$ (iii): For $f\in E$,
% $\langle f, P_E{\psi_i}\rangle = \langle P_Ef, \psi_i\rangle =\langle f, \psi_i\rangle.$
%Hence the result follows.\\
%(i) $\Longleftrightarrow$ (ii):
%From (iii), $T$ is one-to-one on $E$ if and only if $$E=\ol{\Span}  (\mc E_{\mc D}\{P_V \psi_i: i\in I\}),$$ which is further equivalent to by ........
%$$J_E(x)=\overline{\Span}\{P_{J_E(x)}\psi_i(x): i\in I\},\ \mbox{for a.e. }\ x\in X. $$
%So $T$ is one-to-one if and only if $J_E(x)=\overline{\Span}\{P_{J_E(x)}\psi_i(x): i\in I\},\ \mbox{for a.e. }\ x\in X $. Again applying (iii) on $J_E(x)$,  equivalently we can say that for a.e. $x\in X$, $T(x)$ is one-to-one on $J_{ E}(x)$. \\
%
%(ii) $\Longleftrightarrow$ (iv) and (v): Using (iii), $T$ is one-to-one   sampling operator on $V$ if and only if $T(x)$ is one-to-one   sampling operator on $J_{E}(x)$ for a.e. $x\in X$,  which is finite dimensional. Since the range of $T_{\mc A'}^*(x)$  is $E$, (iv) and (v) easily follows.
\end{proof}
In classical sampling theory, the signals to be sampled come from a single space. Practically, signals for sampling should be considered from the union of subspaces. This approach has applications in sparse signal representation, compressed sampling, nonuniform splines, and more. The extension of the sampling problem to a union of subspaces was introduced by Lu and Do \cite{ludo2008samplingoperator} and extended further by M. Anastasio and Cabrelli \cite{cabrelli2009samplingoperator}. In our setup, we are going to take a union of finitely generated multiplication invariant subspaces of $L^2(X;\mathcal{H})$ and discuss the injectivity of sampling operator $T$ on the union.

\textit{{\bf Question:} \label{Q}Corresponding to an arbitrary index set $\Delta$, let $\{\mc N_{\delta}\}_{\delta\in \Delta}$ be the  finitely generated MI spaces in $L^2(X;\mc H)$ and 
	\begin{equation}\label{Q}
		\mc X: =\bigcup _{\delta\in \Delta}\mc N_{\delta}.
	\end{equation}
	When can each element $f\in \mc X$ be uniquely represented by its samples?}

The discussion of injectivity of the sampling operator for a union of MI spaces in the collection $\{\mc N_\delta\}_{\delta \in \Delta}$ comes with a problem. The definition of an injective operator (say $A$)  says that for any $x_1,x_2 $ in the domain of the operator, $Ax_1 = Ax_2$ implies $x_1=x_2.$ The issue while considering the union of subspaces is that the points $x_1$ and $x_2$ can be from different subspaces of the union and exploits the linearity. To deal with the issue, we consider the union of a sum of MI subspaces of $L^2(X;\mc H).$  The idea was originated in \cite{ludo2008samplingoperator} and supported further by \cite{cabrelli2009samplingoperator} where they considered the collection $\{\mc S_\delta\}_{\delta \in \Delta}$ of shift invariant spaces and took the union of a sum of each shift invariant space in the collection. Similarly, we consider the union \begin{equation}\label{chi}
	\mc X: =\bigcup _{\delta\in \Delta}\mc N_{\delta,\theta}
\end{equation}
of the sum of each of the finitely generated MI spaces in the collection $\{\mc N_\delta\}_{\delta \in \Delta}.$ Each space $\mc N_{\delta,\theta}$ in the union has the following form
\begin{equation}\label{Ndeltatheta}
	\mc N_{\delta, \theta}:=\mc N_{\delta}+\mc N_{\theta}=\{u+v: u\in \mc N_{\delta}, v \in \mc N_{\theta}\}
	,\end{equation}
and every vector in the union is a secant vector which has importance in dimensionality reduction \cite{roomhead2000new}.  

Now, we deal with the injectivity of sampling operator $T$ on $\mc X$ given in (\ref{chi}). The conditions for injectivity of the sampling operator for the case of the union of closed subspaces have been passed to those of single space \cite[Proposition 1]{ludo2008samplingoperator}. But here the problem is sum of two MI spaces $\mc N_{\delta, \theta}:=\mc N_{\delta}+\mc N_{\theta}=\{u+v: u\in \mc N_{\delta}, v \in \mc N_{\theta}\}$ need not to be closed and therefore $\mc N_{\delta,\theta}$
need not to be an MI space. This directly shifts the problem into the field of closedness of subspaces where we will take the help of Section \ref{S:Results}. The following Proposition suggests that the closure of the sum of two MI spaces is again an MI space generated by the union of the generators of two MI spaces and thus for the space $\mc N_{\delta,\theta}$ to be an MI space, it is enough for it to be closed. A similar result for the case of shift invariant spaces can be found in \cite{cabrelli2009samplingoperator}.
\begin{prop}\label{Misumclosed}
	For $\mc A, \mc A' \subset L^2(X;\mc H)$, $S_\mc D(\mc A \cup \mc A') = \ol{\mc S_\mc D(\mc A)+ \mc S_\mc D(\mc A')}.$ 
\end{prop}
\begin{proof}
	The proof follows by the  linearity of multiplication operator $M_\phi$ and observing
	\begin{equation*}
		\begin{aligned}
			\ol{\mc S_\mc D(\mc A)+ \mc S_\mc D(\mc A')}& = 
			\ol{\ol{\Span}\{M_\phi \varphi : \varphi \in \mc A, \phi \in \mc D\} + \ol{\Span}\{M_\phi \varphi : \varphi \in \mc A', \phi \in \mc D\}}\\
			&=\ol{\Span}\{M_\phi(\varphi+\varphi') : \varphi \in \mc A, \varphi' \in \mc A', \phi \in \mc D \}.
		\end{aligned}
\end{equation*}\end{proof}The following result is an abstract version of \cite[Theorem 6.5]{cabrelli2009samplingoperator}.
%Also the one-to-one property of an operator on a subspace does not imply the one-to-one on its closure 
 % Given a set $\Psi:= \{\psi_i\}_{i\in I}$ such that $\mc E_{\mc D}(\Psi)$ is a Bessel in $L^2(X;\mc H)$, the corresponding sampling operator is:
%$$A: L^2(X;\mc H)\ra \ell^2( \mc D \times I), Af= \{\langle f, M_\phi \phi_i\rangle\}_{i\in I, \phi \in \mc D}.$$
%Let for each $\delta, \theta \in \Delta$, the subspaces,
%$$\overline{\mc S_{\delta, \theta}}:= \overline{\mc S_{\delta}+\mc S_{\theta}}$$
%It is clear that the closure of the sum of the two finitely generated multiplication invariant space is again multiplication invariant. 

\begin{thm}\label{SamplingMain}
	Let $\mc  A=\{\psi_i\}_{i\in I}$ be a sequence in   $L^2(X;\mc H)$ such that $\mc E_{\mc D}(\mc A)$ is  Bessel and the supremum cosine angle $$\mathfrak S(\mc N_{\delta},\mc N_{\theta})<1,  \mbox{ for every} \  \delta,\theta \in \Delta.$$ If $\mc A_{\delta,\theta}'$ is  a finite set of generators for $\mc N_{\delta,\theta},$  defined in (\ref{Ndeltatheta}), then the sampling operator associated with $\mc E_\mc D(\mc A)$ 
	is one-to-one for $\bigcup _{\delta, \theta \in \Delta}\mc N_{\delta,\theta}$ if and only if
	$\dim (\range(G_{\mc A_{\delta, \theta}', \mc A})(x))=\dim J_{\mc A'_{\delta,\theta}}(x),$ for a.e. $x \in X$, and $\delta,\theta \in \Delta$.
	
\end{thm}
%\begin{cor}[Corollary 5.14]
%Let $\mc E_{\mc D}(\mc A)$ be a Bessel sequence in $L^2(X;
%\mc H)$ for some set of function $\Psi$. Let for every $\delta, \theta \in \Delta$, $\Phi_{\delta, \theta}$ be a finite set of generators for $\overline{S}_{\delta, \theta}$. If for each $\delta, \theta \in \Delta$
%$$\dim (\range (G_{\Phi_{\delta, \theta}}\Psi(\omega)))=\dim \overline{S}_{\delta, \theta} \mbox{for a.e.} x\in X.$$
%Then $A$ is an one-to-oneoperator on $\mc X$.
%\end{cor}

\begin{proof}
Using Theorem (\ref{T:sup-pointwise}), $ \mc N_{\delta, \theta}$ is closed  for each $\delta, \theta\in \Delta$,  and then by applying Proposition \ref{Misumclosed}, it is clear that for each $\delta, \theta \in \Delta$ the space $\mc N_{\delta,\theta}$  is an MI space. Also, the sampling operator $T$ is one-to-one on $\bigcup _{\delta,\theta \in \Delta}\mc N_{\delta,\theta}$ if and only if $T$ is one-to-one on $\mc N_{\delta, \theta}$
for each $\delta,\theta \in \Delta$  \cite[Proposition 1]{ludo2008samplingoperator}. Finally, applying Proposition \ref{SamplingOp1-1} will give us the desired result.
\end{proof}

\section{Application to locally compact group}\label{LC Group}

 Let $\mathscr G$ be a second countable locally compact group which need not to be abelian, and  $\Gamma$ be a closed abelian subgroup of $\mathscr G$. A closed subspace  $V$ in $L^2(\mathscr G)$ is said to be   \textit{$\Gamma$-translation invariant ($\Gamma$-TI)} if $L_\xi f \in  V$  for all $f\in  V$
and $\xi \in \Gamma$, where for each $\eta \in \mathscr G$ the \textit{left translation} $L_\eta$ on $L^2 (\mathscr G)$ is defined by    
$$(L_\eta f)(\gamma) = f(\eta^{-1} \gamma), \quad \gamma \in \mathscr G \  \mbox{and }\ f\in L^2(\mathscr G).$$
Translation invariant spaces are widely used in various domains, significant among them are harmonic analysis, signal processing, and time-frequency analysis.
Researchers are often interested in characterizing the class of generators of TI/shift invariant spaces and its properties that allow the reconstruction of any function/signal/image \textit{via} a reproducing formula. For more details, refer to \cite{bownik2015structure}. 

For a  countable family of functions   $\mathscr A\ss L^2(\mathscr G)$, 
let us consider a \textit{$\Gamma$-translation generated ($\Gamma$-TG)} system $\mc E^{\Gamma}(\mathscr A)$ and   its associated  $\Gamma$-translation invariant ($\Gamma$-TI)  space $\mc S^{\Gamma}(\mathscr A)$  generated by $\mathscr A,$ i.e., 
$$
\mc E^{\Gamma}(\mathscr A) :=\{L_\xi \varphi : \varphi \in \mathscr A, \xi\in \Gamma\}  \quad \mbox{and} \quad  \mc S^{\Gamma}(\mathscr A) := \overline{\Span} \ \mc E^{\Gamma}(\mathscr A),
$$ 
respectively.

For  $x \in  \mathscr G$, a right coset of $\Gamma$ in $\mathscr G$ with respect to $x$ is denoted by $\Gamma x$, and for  a function   $f :\mathscr G \rightarrow \mathbb C$, we define a  complex valued    function $f^{\Gamma x}$ on $\Gamma$ by  
$
f^{ \Gamma x}(\gamma)=f(\gamma \, \Xi(\Gamma x)), \quad     \gamma \in  \Gamma, 
$
where   the   space of orbits        $\Gamma\backslash \mathscr G = \{\Gamma x: x \in  \mathscr G\}$ is    the set of all right cosets of $\Gamma $ in $\mathscr G$, and $\Xi : \Gamma \backslash \mathscr G\ra  \mathscr G$ is a \textit{Borel section} for  the quotient space $\Gamma \backslash \mathscr G$. Then the Fourier transform of $f^{ \Gamma x} \in L^1 (\Gamma)$  is given by $\widehat{f^{ \Gamma x}} (\alpha)=\int_{\Gamma} f^{ \Gamma x} (\gamma) \alpha (\gamma^{-1}) \ d\mu_\Gamma (\gamma),$ for $\alpha \in \widehat{\Gamma},$ 
which can be  extended to $L^2 (\Gamma)$.  The \textit{Zak transformation} $\mathcal Z$ of  $f\in L^2(\mathscr G)$ for the pair $(\mathscr G,\Gamma)$  is   defined by
\begin{equation}\label{zak trans}
	(\mc Zf)(\alpha)(\Gamma x) =\widehat{f^{\Gamma x}}(\alpha), \quad a.e. \quad  \alpha \in \widehat{\Gamma} \ \mbox{and} \   \Gamma x\in \Gamma\backslash \mathscr G,
\end{equation}
which is a    unitary linear transformation from $ L^2(\mathscr G)  $ to $L^2(\widehat{\Gamma}; L^2(\Gamma\backslash \mathscr G))$ \cite{ iverson2015subspaces}.  
Note that the Zak transform $\mc Z$ is closely associated with fiberization map $\mathscr F$ when $\mathscr G$ becomes abelian. For a second countable locally compact abelian (LCA) group $\mc G$ and its closed subgroup $\Lambda$,  the \textit{fiberization}   $\mathscr F$ is a unitary map from $L^2 (\mc G)$ to $L^2 ( \widehat{\mc G}/ \Lambda^\perp ; L^2 (\Lambda^\perp))$ given by 
\begin{align}\label{fib}
	(\mathscr F f)(\beta\Lambda^\perp )(x)=\widehat{f} (x \, \zeta (\beta \Lambda^\perp)), x\in \Lambda^\perp, \beta\in \widehat{\mc G}, 
\end{align}
for $f \in L^2 (\mc G)$, where $\Lambda^\perp:=\{\beta\in \widehat{\mc G}: \beta(\lambda)=1, \ \forall \ \lambda \in \Lambda \}$, $\Lambda^\perp\backslash \widehat{\mc G}=\widehat{\mc G}/ \Lambda^\perp$ and  $\zeta : \widehat{\mc G}/ \Lambda^\perp \rightarrow \widehat{\mc G}$ is Borel section which maps compact sets to pre-compact sets. 	 
The  Zak transform and fiberization map  on the Euclidean space $\mathbb R^n$  by the action of integers $\mathbb Z^n$ are 
$$
({\mc Z} f) (\xi, \eta)=\sum_{k \in \mathbb Z^n} f(\xi+k) e^{-2\pi i k\eta}, \ \mbox{and}  \  (\mathscr Ff)(\xi)(k)=\widehat {f}(\xi+k),  
$$
for $k \in \mathbb Z^n$, $\xi, \eta \in \mathbb  T^n$ and $f \in L^1 (\mathbb R^n) \bigcap L^2 (\mathbb R^n)$.

 The Zak transform or fiberization map allows us to handle TI spaces using MI spaces as both converts translation operators   $L_\gamma$ into multiplication operators $M_\phi$ for some suitable function $\phi \in L^\infty(X)$ i.e., for $f \in  L^2 (\mathscr G)$, 
$	(\mc ZL_{\gamma}f)(\alpha) =  (M_{\phi_\gamma}  \mc Zf) (\alpha), \quad 
\mbox{for    a.e.} \ \alpha \in \widehat{\Gamma}\  \mbox{and} \ \gamma \in \Gamma, $ 
where  $M_{\phi_\gamma}$ is the multiplication operator on $L^2(\widehat{\Gamma}; L^2(\Gamma\backslash \mathscr G))$, $\phi_\gamma (\alpha)=\overline{\alpha(\gamma)}$   and $\phi_\gamma \in L^\infty (\widehat{\Gamma})$ for each $\gamma \in \Gamma$. Therefore, our goal can be established by converting the problem of $\Gamma$-TI space $\mc S^{\Gamma}(\msc A)$ into the MI spaces on $L^2 (X; \mc H)$ with the help of Zak transform, where $X=\widehat{\Gamma}$ and $\mc H=L^2(\Gamma\backslash \mathscr G)$. 
% 
%  So it sufficient to study MI spaces and then the entire setting can be transferred into TI spaces using these maps.
% 
% %%%%%%%%%%%%%
% Observe that    the Zak transform $\mc Z$   satisfies the intertwining property with the left translation and multiplication operators,  i.e., for $f \in  L^2 (\mathscr G)$, 
% $	(\mc ZL_{\gamma}f)(\alpha) =  (M_{\phi_\gamma}  \mc Zf) (\alpha), \quad 
% \mbox{for    a.e.} \ \alpha \in \widehat{\Gamma}\  \mbox{and} \ \gamma \in \Gamma, $ 
% where  $M_{\phi_\gamma}$ is the multiplication operator on $L^2(\widehat{\Gamma}; L^2(\Gamma\backslash \mathscr G))$, $\phi_\gamma (\alpha)=\overline{\alpha(\gamma)}$   and $\phi_\gamma \in L^\infty (\widehat{\Gamma})$ for each $\gamma \in \Gamma$. Therefore, our goal can be established by converting the problem of $\Gamma$-TI space $\mc S^{\Gamma}(\msc A)$ into the MI spaces on $L^2 (X; \mc H)$ with the help of Zak transform, where $X=\widehat{\Gamma}$ and $\mc H=L^2(\Gamma\backslash \mathscr G)$. 
 
 In this setup  the   range function is   $ J:\widehat{\Gamma}\ra\{ {\mbox{closed subspaces of } L^2(\Gamma\backslash \mathscr G)}\}$.
For the $\Gamma$-TI space $\mc S^{\Gamma}(\mathscr A)$  in $L^2 (\mathscr G)$,   the corresponding range function $J$ is such that, for a.e.  $\alpha \in \widehat{\Gamma}$, $J(\alpha)$ is defined by
\begin{equation}\label{J(alpha)forZak}
	J(\alpha)= \overline{\Span}\{(\mc Zf)(\alpha) : f \in \mathscr A_0\} =:  J_{\mathscr A} (\alpha).
\end{equation}
One of the benefits of Zak transform for the pair $(\G, \Gamma)$ is that the various inaccessible examples pairs like        $(\mathbb R^n,  \mathbb R^m)$, $(\mathbb R^n, \mathbb Z^m)$,  $(\mathbb Q_p, \mathbb Z_p)$, $(\mathcal G, \Gamma)$, etc., can be accessed through it where  $n\geq m$,  $\Gamma$ (not necessarily co-compact, i.e., $\mathcal G/\Gamma$-compact, or uniform lattice) is a  closed subgroup of the second countable locally compact abelian (LCA) group $\mathcal G$, and   $\mathbb Z_p$ is the  $p$-adic integer in the $p$-adic number $\mathbb Q_p$.

Since $\mc Z: L^2(\mathscr G)\ra L^2(\widehat \Gamma; \Gamma\backslash \mathscr G)$ is an unitary operator, we have
\begin{equation}\label{E: AngleMITI}
	\begin{aligned}
		\mathfrak S(\mc S^{\Gamma}(\msc A), \mc S^{\Gamma}(\msc A'))&=\sup\left\{ \frac{|\langle  u,  v\rangle|}{\|u\| \|v\|}:u \in \mc S^{\Gamma}(\msc A)\backslash \{0\}, v\in \mc S^{\Gamma}(\msc A')\backslash \{0\} \right\}\\
		&=\sup\left\{ \frac{|\langle \mc Z u, \mc Zv\rangle|}{\|\mc Zu\| \|\mc Zv\|}:\mc Zu \in \mc Z\mc S^{\Gamma}(\msc A)\backslash \{0\}, \mc Zv\in \mc Z\mc S^{\Gamma}(\msc A')\backslash \{0\} \right\}\\
		&=\mathfrak S(\mc Z\mc S^{\Gamma}(\msc A) ,\mc Z\mc S^{\Gamma}(\msc A'))
		.\end{aligned} 
\end{equation}
The relation (\ref{E: AngleMITI}) of the angle between two MI spaces and the corresponding TI spaces will allow us to have results similar to Theorem (\ref{T:sup-pointwise}) and Theorem (\ref{T:closedness}) for the set up of TI spaces. 
The next theorem is an application of Theorem (\ref{T:sup-pointwise}) which provides a  necessary and sufficient condition for the sum of two translation invariant spaces of a locally compact group to be closed. A similar result has been proved for shift invariant spaces \cite[Theorem 3.6]{kim2006supremum}.
\begin{thm}\label{T:LCClosed} 
	Let  $\mathscr A=\{\varphi_i: i\in I\}$ and $\mathscr A'=\{\psi_i: i\in I\}$ be two countable collections of functions in $L^2(\mathscr G).$ Define $\sigma(\mc S^{\Gamma}(\msc A))=\{\alpha \in \widehat \Gamma: J_{\msc A}(\alpha)\neq \{0\}\}$
	and $\Omega:=\sigma(\mc S^\Gamma(\mathscr A)) \cap \sigma(\mc S^\Gamma(\mathscr A')).$ Then the supremum cosine angle between $\mc S^\Gamma(\mathscr A)$ and $\mc S^\Gamma(\mathscr A')$ is given by 
	$$\mathfrak S(\mc  S^\Gamma(\mathscr A), \mc S^\Gamma(\mathscr A'))=\ess-sup_{\alpha \in \Omega}\{\mathfrak S(J_{\mathscr A}(\alpha), J_{\mathscr A'}(\alpha))\}.$$
	In addition, if the translation generated systems  $\mc E^\Gamma(\mathscr A)$ and $\mc E^{\Gamma}(\mathscr A')$ are  frames for $\mc S^\Gamma(\mathscr A)$ and $\mc S^\Gamma(\mathscr A')$, respectively, then
	$$\mathfrak S(\mc  S^\Gamma(\mathscr A), \mc S^\Gamma(\mathscr A'))={\ess-sup}_{\alpha\in \Omega}\left\|(G_{\mathscr A'}(\alpha)^\dagger )^\frac{1}{2} G_{\mathscr A, \mathscr A'}(\alpha) (G_{\mathscr A}(\alpha)^\dagger)^\frac{1}{2}\right\|,$$ where $\dagger$ denotes the pseudoinverse.
\end{thm}	
Similarly, the next result is an application of Theorem (\ref{T:closedness}) to TI spaces.		
\begin{thm}
	In addition to the standing assumptions as in Theorem (\ref{T:LCClosed}), let $$\Omega':=\left\{\alpha\in \Omega: J_{\mathscr A}(\alpha)\cap J_{\mathscr A'}(\alpha)= \{0\}\right\}
	=\Omega\backslash  \sigma\left( \mc S^\Gamma(\mathscr A)\cap \mc S^\Gamma(\mathscr A')\right).$$
	Then the following are equivalent:	
\end{thm}	
\begin{enumerate}
	\item[(i)]   $\mc S^\Gamma(\mathscr A)|_{\Omega'}+\mc S^\Gamma(\mathscr A')|_{\Omega'}$ is closed. 
	\item[(ii)] $\mathfrak S(\mc S^\Gamma(\mathscr A)|_{\Omega'},\mc S^\Gamma(\mathscr A')|_{\Omega'})<1$.
	\item[(iii)] $J_{\mathscr A}(\alpha)+J_{\mathscr A'}(\alpha)$ is closed for 
	a.e. $\alpha \in  \Omega'$.
	\item[(iv)]  $\mathfrak S(J_{\mathscr A}(\alpha), J_{\mathscr A'}(\alpha))<1$ for 
	a.e. $\alpha \in  \Omega'$.
\end{enumerate} 
In the case of LCA group $\mc G$ and its closed abelian subgroup $\Gamma,$ we can state the above theorems using the fiberization map (\ref{fib}). For this case, the equation (\ref{E: AngleMITI}) will be modified as follows
$$\mathfrak S(\mc S^{\Gamma}(\msc A), \mc S^{\Gamma}(\msc A')) = \mathfrak S(\mathscr {F}\mc S^{\Gamma}(\msc A) ,\mathscr F \mc S^{\Gamma}(\msc A')).$$
Similar to MI spaces, here also we can talk about the injectivity of the sampling operator associated to TI space $\mc S^{\Gamma}(\msc A).$
Let us first give some definitions.
\begin{defn}
	Let $\mathscr A = \{\psi_i\}_{i \in I} \subset L^2(\mathscr G)$ be a countable collection  such that the $\Gamma$-TG system $\mc E^{\Gamma}(\mathscr A)$ is Bessel. Then the \textit{sampling operator} assocated with the system $\mc E^{\Gamma}(\mathscr A)$ is defined by
	\begin{equation}
		T: L^2(\mathscr{G})  \ra L^2( \Gamma \times I),\ Tf:= \{\langle f,L_\xi \psi_i\rangle\}_{\xi \in \Gamma, i \in I}.	
	\end{equation}	
\end{defn}
Corresponding to a arbitrary index set $\Delta$, let $\{\mc S_{\delta,}\}_{\delta\in \Delta}$ be a collection of  finitely generated TI spaces  in $L^2(\mathscr G)$. Let $$\mc X: =\bigcup _{\delta\in \Delta}\mc S_{\delta,\theta},$$ 
where $\mc S_{\delta, \theta}:=\mc S_{\delta}+\mc S_{\theta}=\{u+v: u\in \mc S_{\delta}, v \in \mc S_{\theta}\}.$
The next result is an application of Theorem (\ref{SamplingMain}) for translation invariant spaces.
\begin{thm}
	Let $\msc  A=\{\psi_i\}_{i\in I}$ be a countable collection in  $L^2(\msc G)$ such that $\mc E^\Gamma(\msc A)$ is  Bessel and the supremum cosine angle $$\mathfrak S(\mc S_{\delta},\mc S_{\theta})<1 \mbox{ for every} \  \delta,\theta \in \Delta.$$ If $\msc A_{\delta,\theta}'$ is  a finite set of generators for $\mc S_{\delta,\theta},$  then the sampling operator associated with $\mc E^\Gamma(\msc A)$ is one-to-one for  $\bigcup _{\delta\in \Delta}\mc S_{\delta,\theta}$ if and only if
	$\dim (\range(G_{\msc A_{\delta, \theta}', \msc A})(\alpha))=\dim J_{\mc A'_{\delta,\theta}}(\alpha)$ for a.e. $\alpha \in \widehat \Gamma$,  and $\delta,\theta \in \Delta$.
	
\end{thm}

\begin{example}
	For a prime number $p$, let $\mathbb Q_p$ be the group of $p$-adic numbers and $\mathbb Z_p$ be its closed subgroup. It is a locally compact abelian group. All its subgroups are compact and open. Let the fundamental domain for $\mathbb Z_p$  is $\Delta$ which is discrete. The Zak transform for $( \mathbb Q_p, \mathbb Z_p)$ is $\tilde{\mc Z} f(x, y) =\int_{\mathbb Z_p}f(y+\xi)e^{-2 \pi i x \xi}\ d{\mu_{\mathbb Z_p}}(\xi)$, for $x, y \in \Delta$.  We can find the supremum cosine angle between two $\mathbb Z_p$-invariant subspaces by the Theorem \ref{T:sup-pointwise}.
\end{example}
\begin{example}
	Let $\mc G=\mathbb R^n$ and $\Lambda=\mathbb Z^n$. Then, $\widehat{\mc G}=\mathbb R^n$, $\Lambda^\perp=\mathbb Z^n$ and the fundamental domain for  $\mathbb Z^n$ is $\widehat{\mc G}\backslash \Lambda^\perp=\mathbb T^n$. Then, the fiberization map  $ \mathscr F:L^2(\mathbb R^n)\ra L^2(\mathbb T^n;\ell^2(\mathbb Z^n))$  is defined by
	$\mathscr Ff(\xi)=\{\widehat {f}(\xi+k)\}_{k\in \mathbb Z^n}, \xi \in \mathbb  T^n$. Therefore by considering  countable families $\msc A$ and $\mathscr A'$ in $L^2 (\mathbb R^n)$,  we can find  supremum cosine angle between  $\mc S^\Lambda(\msc A)$ and $\mc S^\Lambda(\msc A')$ and discuss their closedness Theorem (\ref{T:LCClosed}).
\end{example}
\providecommand{\bysame}{\leavevmode\hbox to3em{\hrulefill}\thinspace}
\providecommand{\MR}{\relax\ifhmode\unskip\space\fi MR }
\providecommand{\MRhref}[2]{%
	\href{http://www.ams.org/mathscinet-getitem?mr=#1}{#2}
}
\providecommand{\href}[2]{#2}


\begin{thebibliography}{10}

	\bibitem{aldroubi1998construction}
	Akram Aldroubi, Patrice Abry, and Michael Unser, \emph{Construction of
		biorthogonal wavelets starting from any two multiresolutions}, IEEE
	transactions on signal processing \textbf{46} (1998), no.~4, 1130--1133.
	
	\bibitem{cabrelli2009samplingoperator}
	Magal\'{\i} Anastasio and Carlos Cabrelli, \emph{Sampling in  a union of frame
		generated subspaces}, Sampl. Theory Signal Image Process. \textbf{8} (2009),
	no.~3, 261--286. 
	

	\bibitem{bownik2015structure}
	M.~Bownik and K.~A. Ross, \emph{The structure of translation-invariant spaces
		on locally compact abelian groups}, J. Fourier Anal. Appl. \textbf{21}
	(2015), no.~4, 849--884.
	
	\bibitem{bownik2004biorthogonal}
	Marcin Bownik and Gustavo Garrig{\'o}s, \emph{Biorthogonal wavelets, MRA's and
		shift-invariant spaces}, Studia Mathematica \textbf{160} (2004), 231--248.
	
		\bibitem{bownik2019multiplication}
	M. Bownik and J.~W. Iverson, \emph{Multiplication-invariant operators
		and the classification of {LCA} group frames}, J. Funct. Anal. \textbf{280}
	(2021), no.~2, Paper No. 108780, 59.
	\bibitem{christensen2016introduction}
	O.~Christensen, \emph{An introduction to frames and Riesz bases}, Springer,
	2016.
	\bibitem{conway1997course}
	J.B. Conway, \emph{A course in functional analysis}, Springer, vol.~96, 1997.
	 
	\bibitem{roomhead2000new}Broomhead, David S and Kirby, Michael, \emph{A new approach to dimensionality reduction: Theory and algorithms}, SIAM Journal on Applied Mathematics.
	\textbf{60}(2000),
	2114--2142.
	\bibitem{helsonlectures}
	H.~Helson, \emph{Lectures on invariant subspaces (academic press, new york,
		1964).}, Helson Lectures on invariant subspaces1964.
	
	\bibitem{iverson2015subspaces}
	J.~W. Iverson, \emph{Subspaces of $L^2(G)$ invariant under translation by an
		abelian subgroup}, J. Funct. Anal. \textbf{269} (2015), no.~3, 865--913.
	

	

	
	\bibitem{Kim2002onrieszwavelet}
	Hong~Oh Kim, Rae~Young Kim, Y.H. Lee, and Jae~Kun Lim, \emph{On riesz wavelets
		associated with multiresolution analyses}, Appl. Comput. Harmon. Anal.
	\textbf{13} (2002), 138--150.
	
	\bibitem{kim2003quasi}
	Hong~Oh Kim, Rae~Young Kim, and Jae~Kun Lim, \emph{Quasi-biorthogonal frame
		multiresolution analyses and wavelets}, Advances in Computational Mathematics
	\textbf{18} (2003), no.~2, 269--296.
	
	\bibitem{kim2005infimum}
Hong~Oh Kim, Rae~Young Kim, and Jae~Kun Lim,\emph{The infimum cosine angle between two finitely generated
		shift-invariant spaces and its applications}, Applied and Computational
	Harmonic Analysis \textbf{19} (2005), no.~2, 253--281.
	
	\bibitem{kim2006supremum}
	Hong~Oh Kim, Rae~Young Kim, and Jae~Kun Lim, \emph{Characterization of the closedness of the sum of two
		shift-invariant spaces}, Journal of mathematical analysis and applications
	\textbf{320} (2006), no.~1, 381--395.
	
	\bibitem{kim2008internal}
Hong~Oh Kim, Rae~Young Kim, and Jae~Kun Lim, \emph{Internal structure of the multiresolution analyses defined by
		the unitary extension principle}, Journal of Approximation Theory
	\textbf{154} (2008), no.~2, 140--160.
	
	\bibitem{ludo2008samplingoperator}
	Y.~Lu and M.~Do, \emph{A theory for sampling signals from a union of
		subspaces}, IEEE Trans. Signal Process. \textbf{56} (2008), no.~6,
	2334--2345.
	\bibitem{unser1994general}
	Unser, Michael and Aldroubi, Akram, \emph{A general sampling theory for nonideal acquisition devices}, IEEE Transactions on Signal Processing. \textbf{42}(1994), 2915--2925.
	\bibitem{Tang2000obliqueprojection}
	W.-S. Tang, \emph{Oblique projections, biorthogonal Riesz bases and
		multiwavelets in Hilbert spaces}, Proc. Amer. Math. Soc. \textbf{128} (2000),
	463--473.
\end{thebibliography}
\end{document}